\documentclass{amsart}
\usepackage{amssymb,amsmath,amsthm,amsxtra}
\usepackage{hyperref}
\usepackage[all]{xy}
\usepackage[usenames]{color}
\usepackage{amscd}
\usepackage{amsfonts}
\usepackage{amssymb}
\usepackage{mathrsfs}
\usepackage{url}
\usepackage{enumerate}
\usepackage[usenames,dvipsnames]{xcolor}

\definecolor{dblue}{RGB}{6,69,173}
\definecolor{lblue}{RGB}{11,0,128}

\newcommand{\colorlinks}{true}

\newcommand{\linkcolor}{lblue}
\newcommand{\citecolor}{green}
\newcommand{\urlcolor}{dblue}

\newcommand{\linkbordercolor}{red}
\newcommand{\citebordercolor}{green}
\newcommand{\urlbordercolor}{cyan}

\usepackage{hyperref}
\hypersetup{colorlinks=\colorlinks,linkbordercolor=\linkbordercolor, linkcolor=\linkcolor, citecolor=\citecolor, urlcolor=\urlcolor,urlbordercolor=\urlbordercolor,citebordercolor=\citebordercolor}%

\let\oldtocsubsection=\tocsubsection
\let\oldtocsubsubsection=\tocsubsubsection
\renewcommand{\tocsubsection}[2]{\hspace{1em}\oldtocsubsection{#1}{#2}}
\renewcommand{\tocsubsubsection}[2]{\hspace{2em}\oldtocsubsubsection{#1}{#2}}

\newcommand{\hrefHid}[2]{
\hypersetup{urlbordercolor={1 1 1}}%
\hypersetup{urlcolor=black}%
\href{#1}{#2}%
\hypersetup{urlbordercolor=\urlbordercolor}%
\hypersetup{urlcolor=\urlcolor}%
}

\newcommand{\inhref}[2]{\hyperref[#1]{#2}}

\newcommand{\inhrefHid}[2]{%
\hypersetup{linkbordercolor={1 1 1}}%
\hypersetup{linkcolor=black}%
\inhref{#1}{#2}%
\hypersetup{linkbordercolor=\linkbordercolor}%
\hypersetup{linkcolor=\linkcolor}%
}

\newcommand{\defHref}[3]{\newcommand{#1}[1][#3]{\href{#2}{##1}}}
\newcommand{\defInhref}[3]{\newcommand{#1}[1][#3]{\inhref{#2}{##1}}}
\newcommand{\defHrefHid}[3]{\newcommand{#1}[1][#3]{\hrefHid{#2}{##1}}}
\newcommand{\defInhrefHid}[3]{\newcommand{#1}[1][#3]{\inhrefHid{#2}{##1}}}

\newcommand{\defHrefBoth}[3]{%
\expandafter\defHrefHid \csname #3Hid\endcsname {#1}{#2}%
\expandafter\defHref \csname #3Vis\endcsname {#1}{#2}%
}
\newcommand{\defInhrefBoth}[3]{%
  \expandafter\defInhrefHid \csname #3Hid\endcsname {#1}{#2}%
  \expandafter\defInhref \csname #3Vis\endcsname {#1}{#2}%
}

\newcommand{\defHrefBothVis}[3]{%
\defHrefBoth{#1}{#2}{#3}%
\expandafter\defHref \csname #3\endcsname {#1}{#2}%
}
\newcommand{\defInhrefBothVis}[3]{%
  \defInhrefBoth{#1}{#2}{#3}%
  \expandafter\defInhref \csname #3\endcsname {#1}{#2}%
}

\newcommand{\defHrefBothHid}[3]{%
\defHrefBoth{#1}{#2}{#3}%
\expandafter\defHrefHid \csname #3\endcsname {#1}{#2}%
}
\newcommand{\defInhrefBothHid}[3]{%
  \defInhrefBoth{#1}{#2}{#3}%
  \expandafter\defInhrefHid \csname #3\endcsname {#1}{#2}%
}
\newcommand{\et}{{\'{e}tale }}

\newcommand{\term}[2]{%
\label{#2}%
\emph{#1}%
 \globaldefs =1%
\defInhrefBothHid{#2}{#1}{#2}%
 \globaldefs =0%
}

\newcommand{\mailto}[1]{\href{mailto:#1}{\nolinkurl{#1}}}
\let \orgemail=\email
\let \orgurladdr=\urladdr
\renewcommand{\email}[1]{\orgemail{\mailto{#1}}}
\renewcommand{\urladdr}[1]{\orgurladdr{\url{#1}}}

\baselineskip 18pt \textwidth 16cm  \theoremstyle{plain}

\newtheorem{thm}{Theorem}[subsection]
\newtheorem{prop}[thm]{Proposition}
\newtheorem{lem}[thm]{Lemma}
\newtheorem{defn}[thm]{Definition}
\newtheorem{notn}[thm]{Notation}
\newtheorem{cor}[thm]{Corollary}

\newtheorem{rem}[thm]{Remark}

\newtheorem*{theorem*}{Theorem}
\newtheorem*{remark*}{Remark}
\newtheorem{lemma}[thm]{Lemma}
\newtheorem{lemma*}{Lemma}
\newtheorem{proposition}[thm]{Proposition}
\newtheorem{remark}[thm]{Remark}
\newtheorem{rmk}[thm]{Remark}

\newtheorem{theorem}[thm]{Theorem}
\newtheorem{definition}[thm]{Definition}
\newtheorem{notation}[thm]{Notation}

\newtheorem{corollary}[thm]{Corollary}

\newtheorem{conjecture}[thm]{Conjecture}
\newtheorem*{conjecture*}{Conjecture}

\newtheorem{introtheorem}{Theorem}

\newtheorem{romtheorem}{Theorem}

\newtheorem{romconj}[romtheorem]{Conjecture}

\newcommand{\om}

\newcommand{\cB}{\mathcal B}
\newcommand{\cC}{\mathcal C}

\newcommand{\cF}{\mathcal F}

\newcommand{\cH}{\mathcal H}
\newcommand{\cI}{\mathcal I}
\newcommand{\cJ}{\mathcal J}

\newcommand{\cL}{\mathcal L}
\newcommand{\cM}{\mathcal M}

\newcommand{\cO}{\mathcal O}

\newcommand{\cR}{\mathcal R}

\newcommand{\cX}{\mathcal X}

\newcommand{\cZ}{\mathcal Z}

\newcommand{\bC}{\mathbb C}

\newcommand{\bG}{\mathbb G}

\newcommand{\bN}{\mathbb N}

\newcommand{\bP}{\mathbb P}

\newcommand{\bR}{\mathbb R}

\newcommand{\fg}{\mathfrak g}

\newcommand{\fI}{\mathfrak I}

\newcommand{\fR}{\mathfrak R}

\newcommand{\fZ}{\mathfrak Z}

\newcommand{\C}{\mathbb C}
\newcommand{\cc}{\mathbb C}

\newcommand{\bmat}{\left( \begin{matrix}}
\newcommand{\emat}{\end{matrix}\right)}

\newcommand{\lbl}[1]{\label{#1}}
\newcommand{\Rami}[1]{{{#1}}}
\newcommand{\RamiA}[1]{{{#1}}}
\newcommand{\RamiB}[1]{{{#1}}}

\newcommand{\RamiC}[1]{{{#1}}}
\newcommand{\RamiD}[1]{{{#1}}}
\newcommand{\RamiE}[1]{{{#1}}}
\newcommand{\RamiF}[1]{{{#1}}}
\newcommand{\EitanC}[1]{{{#1}}}

\newcommand{\NextVer}[1]{}

\DeclareMathOperator{\SL}{SL}
\DeclareMathOperator{\GL}{GL}

\DeclareMathOperator{\Hom}{Hom}
\DeclareMathOperator{\End}{End}
\DeclareMathOperator{\Spec}{Spec}
\DeclareMathOperator{\spec}{Spec}

\newcommand{\Z}{\mathbb{Z}}

\newcommand{\Sc}{{\mathcal S}}
\newcommand{\Span}{{\operatorname{Span}}}
\renewcommand{\dim}{{\operatorname{dim}}}
\newcommand{\gr}{{\operatorname{gr}}}
\newcommand{\Ind}{\operatorname{Ind}}
\newcommand{\Supp}{{\operatorname{Supp}}}
\renewcommand{\Hom}{{\operatorname{Hom}}}

\newcommand{\CH}{\mathfrak{X}}
\newcommand{\MinPar}{B}
\newcommand{\MinLev}{T}

\makeatletter 
\renewcommand\p@enumii{}
\makeatother

\newcommand{\irr}{irreducible}
\newcommand{\rep}{representation}

\begin{document}

\author{Avraham Aizenbud}
\address{Avraham Aizenbud,
 Faculty of Mathematics and Computer Science, The Weizmann Institute of Science, ISRAEL.%
}
\email{aizner@gmail.com}
\urladdr{http://www.wisdom.weizmann.ac.il/~aizenr/}
\author{Eitan Sayag}
\address{Eitan Sayag, Department of Mathematics, Ben-Gurion University of the Negev, ISRAEL}
\email{sayage@math.bgu.ac.il}
\urladdr{http://www.math.bgu.ac.il/~sayage/}
\title{A short proof of Hironaka's Theorem on freeness of some Hecke modules}
\date{\today}

\maketitle

\begin{abstract}
\RamiD{Let $E/F$ be an unramified extension of non-archimedean local fields of \Rami{residual} characteristic \Rami{different than $2$}. 
 }

We provide a simple \RamiD{geometric} proof of a \RamiD{variation of a} result of Hironaka (\cite{Hir}). \Rami{Namely we prove} that the module $\Sc(X)^{K_0}$ \RamiD{is free} over \RamiD{the Hecke algebra} \RamiE{$\cH(SL_{n}(E),SL_{n}(O_E))$, where $X$ is the space of \RamiD{unimodular} Hermitian forms on \RamiE{$E^n$}  and $O_E$ is the ring of integers in $E$}.
\end{abstract}

 \tableofcontents

\section{Introduction}
Let $F$ be a non-archimedean local field and let $G$ be a reductive $F$-group. Suppose \RamiE{that} $X$ is an algebraic variety equipped with a $G$-action.
Harmonic analysis on the $G(F)$-space $X(F)$, aims to study and decompose certain spaces of functions \RamiE{on  $X(F)$} into simpler representations of $G(F).$
%

\RamiD{A possible approach to this} problem is to consider the structure of the $\cH(G,K)-$ module $\RamiD{\Sc}(X)^{K}$ of $K$-invariant \RamiE{compactly supported} functions on $X$, where $K$ is a compact open subgroup of $G(F)$ and $\cH(G,K)$ is the Hecke algebra of $G(F)$ with respect to the subgroup $K.$

In the special case where \RamiD{$K=K_0$} is a maximal compact subgroup of $G$, \RamiE{the algebra $\cH(G,K)$ is, \RamiD{by Satake's theorem,  a finitely generated polynomial algebra}. Thus, it is natural to study the structure of the module $\Sc(X)^{K_0}$ over this algebra using the language of commutative algebra.} 
It turns out that in many cases, this \RamiE{module is  free,} a result with applications to multiplicities (see \cite{Sa}).
\RamiD{M}any special cases where studied (\cite{Off}, \cite{Hir}, \cite{Mao-Rallis}) and general results are obtained in \cite{Sa} and \cite{SaSph}.




\RamiD{In this paper we prove the following result.}
\RamiE{
\begin{introtheorem}\label{thm:Intro.Hir}
Let $E/F$ be an unramified quadratic extension of local non-archimedean fields of residual characteristic different than $2$. Let $G=SL_{n}(E)$ and $X$ be the space of Hermitian forms on $E^n$ with determinant $1$.  Let $K_0$ be a maximal compact subgroup. Then $\Sc(\Rami{X})^{K_0}$ is a free $\cH(G,K_0)$ module \RamiD{of rank $2^{\dim(V)-1}.$}
\end{introtheorem}
}
\RamiD{
\begin{remark}
In \cite{Hir} a version of the above theorem concerning $GL(V)$ instead of $SL(V)$ was proven. It is not difficult to show that those two versions are equivalent.
\end{remark}
}
The proof in \cite{Hir} was spectral in that it was based on the explicit determination of the spherical functions on the space $X$ associated to unramified representations.
In our approach the proof is based solely on the geometry of the spherical space $X$ and on the analysis of $K_0$ orbits.

\subsection{Idea of the proof}
 The proof is based on a reduction technique we learned from \cite{BL} regarding filtered modules over filtered algebras. This technique allows to deduce the freeness of a module from the freeness of its associated graded. While classically one studies $\mathbb{Z}$-filtered modules, we need to adapt the technique to the case of $\mathbb{Z}^{n}$-filtered modules.

The filtrations we use on the \RamiD{s}pherical Hecke algebra and the spherical Hecke module $\RamiD{\Sc}(X)^{K_0}$ are obtained from \RamiD{Cartan decompositions}.

\subsection{Possible generalizations}
\RamiD{One can not expect that the conclusion of the Theorem holds for any spherical space. Nevertheless, we expect that for a large class of spherical spaces, one can find a subalgbera $B$ of $\cH(G,K_0)$  over which the module $\Sc(X(F))^{K_0}$ is free.}

Our proof of Theorem \ref{thm:Intro.Hir} is based on certain geometric properties that we expect to holds for many symmetric spaces. Informally, we used the fact that the symmetric space $X$ admits a nice Cartan decomposition. More precisely, we use a collection $\{g_{\lambda}\ | \lambda \in \Lambda^{++}\} \subset G$ and a collection $\{x_{\lambda}\ | \delta \in \Delta^{++}\} \subset X$, where $\Lambda^{++} \subset \Lambda$\ is a Weyl chamber of the coweight lattice $\Lambda$\ and similarly for $\Delta^{++} \subset \Delta$ with the following properties:

\begin{itemize}
\item  $G=\bigsqcup_{\lambda \in \Lambda^{++}}K_0g_{\lambda}K_0$
\item $X=\bigsqcup_{\delta \in \Delta^{++}}K_0 \cdot x_{\delta}$
\item $K_0g_{\lambda}K_0 \cdot K_0g_{\mu}K_0=\bigsqcup_{w \in W_{\Lambda}}K_0g_{[w(\lambda)+\mu]}K_0$
where $\{[\gamma]\}:=\left(W_{\Lambda}\cdot \gamma\right) \cap \Lambda^{++}$
\item $ K_0g_{\lambda}K_0 \cdot K_0x_{\delta}=\bigsqcup K_0 \cdot x_{[s(\lambda)+\mu]}$where $\{[\gamma]\}:=\left(W_{\Delta}\cdot \gamma\right) \cap \Delta^{++}$ and $s:\Lambda \to \Delta$ is a certain symmetrization map.
\end{itemize}

We expect that under the above conditions, and certain technical conditions on the lattices $\Delta,\Lambda$, it will be possible to adapt our argument to hold  for any such $X$. In view of \cite{SaSph} we expect those conditions to hold in many cases, but not for all symmetric pairs

\subsection{Acknowledgments}: We would like to thank Omer Offen and Erez Lapid for conversations on \cite{FLO} that motivated our interest in this problem. Part of the work on this paper was done during the research program {\it Multiplicities in representation theory} at the HIM.

\section{Filtered modules and algebras}
We first \RamiD{fix some} terminology regarding filtered modules and algebras.

\begin{definition}$ $
\begin{itemize}
\item \RamiD{F}or $i,j \in \Z^n$ we say that $j\leq i$ if \RamiE{$i-j\in \Z_{\geq 0}^n:=(\Z_{\geq 0})^n$.}
\item \RamiD{B}y a $\Z_{}^n$-filtration on a vector space $V$ we mean a collection of subspaces $F_i(V) \subset V$ for $i \in \Z_{}^n$  s.t. there exist a $\Z_{}^n$\RamiD{-}grading $V =\bigoplus_{i\in \Z_{}^n} F^0_i(V)$ with $F_i(V)=\bigoplus_{j\leq i} F^0_j(V).$

\item For \RamiE{a} $\Z_{}^n$-filtrated vector space \RamiE{$V$,} we denote $Gr^{i}_{F}(V):=F_{i}(V)/\sum_{j<i} F_{j}(V),$ and
 $Gr_{F}(V):= \bigoplus Gr^{i}_{F}(V)$.

\item \RamiD{A}
$\Z_{}^n$-filtration on  an algebra $A$ is a $\Z_{}^n$-filtration $F^i(A)$ on the underlying vector space such that $F_{i}(A)F_{j}(A) \subset F_{i+j}(A).$ Note  that in such a case $Gr_{F}(A)$ is  $\Z_{}^n$-graded algebra.
\item Let $\phi:\Z^n\to \Z^m$ be a morphism. \RamiD{Let $(A,F^{0})$ be $\Z^n$-graded algebra. A} $\phi$-grading on an $A$-module $M$ is a $\Z^m$\RamiD{-}grading $G^0_i(M)$ on the underlying vector space \RamiD{M} such that $F^0_{i}(A)G^0_{j}(M) \subset G^0_{\phi(i)+j}(M).$
\item Let $\phi:\Z^n\to \Z^m$ be a morphism and let $(A,F)$ be \RamiE{a} $\Z^n$ filtrated algebra. A $\phi$-filtration on an $A$-module $M$ is a $\Z^m$-filtration $G_i(M)$ on the underlying vector space such that $F_{i}(A)G_{j}(M) \subset G_{\phi(i)+j}(M).$ Note  that in such a case $Gr_{G}(M)$ is  a $\phi$-graded module over $Gr_{F}(A)$.


\end{itemize}
\end{definition}

The following is an adaptation of a trick we learned from \cite{BL}
(see Lemma 4.2).


\begin{prop}\label{grfree}
Let $\phi:\Z^n\to \Z^m$ be a morphism.

Let $(M,G)$ be a $\RamiB{\mathcal{\phi}}$-filtered module over a $\Z^n$-filtered commutative algebra $(A,F)$.  \RamiD{Assume that for any $i\notin \Z_{\geq 0}^n$ we have $Gr^{i}_{F}(A)= 0$ and  for any $i\notin \Z_{\geq 0}^m$ we have $Gr^{i}_{G}(M)= 0$.}
Suppose that  $Gr_{G}(M)$ is finitely generates free graded module over $Gr_{F}(A)$ (i.e. there exists finitely many homogenous elements that freely \RamiB{generate} $Gr_{G}(M)$).
Then $M$ is a finitely generated free $A$-module.

More specifically if $\bar m_{1},\dots, \bar m_{k} \in Gr_{G}(M)$ are homogenous elements that freely generates   $Gr_{G}(M)$  over $Gr_{F}(A)$, then any lifts  $m_{1},\dots,  m_{k} \in M$ freely generates  $M$  over $A$.

\end{prop}
 \begin{proof}$ $
\begin{enumerate}[{Step 1.}]
\item Proof in the case $m=n=1$, $\phi=id$.\\
See, the proof of \cite[Lemma 4.2]{BL}.
\item Proof in the case $\phi=id$.\\
The proof is by induction on $n$.
Let $\bar m_{1},\dots, \bar m_{k} \in Gr_{G}(M)$ be homogenous elements that freely generates  $Gr_{G}(M)$  over $Gr_{F}(A)$ and  $m_{1},\dots,  m_{k} \in M$ be there lifts.

 For $i \in \Z$, we let $\bar{F}_{i}(A)=\sum_{k \in \Z^{(n-1)}} F_{(i,k)}(A).$  Similarly, we define $\bar{G}_{i}(M)=\sum_{k \in \Z^{(n-1)}} G_{(i,k)}(M).$
These are $\Z$-filtrations. Set $n_{1},\dots,  n_{k} \in Gr_{\bar{G}}(M)$ to be the reductions of $m_{1},\dots,  m_{k} \in M$.

By step 1 it is enough to show that $Gr_{\bar{G}}(M)$ is freely generated by $n_{1},\dots,  n_{k}$ over  $Gr_{\bar{F}}(A)$. For this, define a $ \Z^{(n-1)}$-filtrations on $Gr_{\bar{F}}(A)$ and  $Gr_{\bar{G}}(M)$ by $\widetilde{F}_{j}(Gr^{i}_{\bar{F}}(A))=F_{(i,j)}(A)/F_{(i,j)}(A) \cap \bar{F}_{i-1}(A)$ and  $\widetilde{G}_{j}(Gr^{i}_{\bar{G}}(M))=G_{(i,j)}(M)/G_{i,j}(M) \cap \bar{G}_{i-1}(M).$
\RamiD{The existence of the gradings $F_{i}^0(A), G_{i}^0(M)$  implies that $Gr_{\widetilde{F}}(Gr_{\bar{F}}(A)) \cong Gr_{F}(A)$ and $Gr_{\widetilde{G}}(Gr_{\bar{G}}(M)) \cong Gr_{G}(M).$
Furthermore,   $\bar m_{1},\dots, \bar m_{k}$ are the $\widetilde{G}$-reductions of $n_{1},\dots,  n_{k}$.} Thus, the induction hypothesis implies that $Gr_{\bar{G}}(M)$ is freely generated by $n_{1},\dots,  n_{k}$ over  $Gr_{\bar{F}}(A)$.
\item  The general case.\\ Define $\Z^m$-filtration on $A$ by $\bar F_j(A)=\sum_{i\in \phi^{-1}(j)} F_j(A)$. By step 2, it is enough to show that $Gr_{{G}}(M)$ is freely generated by $\bar m_{1},\dots, \bar m_{k}$ over $Gr_{\bar{F}}(A).$ For this we choose a gradation $F_i^0$  s.t. $F_i(A)=\bigoplus_{j\leq i} F^0_j(A).$ This gives us a linear isomorphism $\psi:Gr_{\bar{F}}(A) \to Gr_{{F}}(A)$ s.t. $\psi(a)m=am$. \RamiD{We note that $\psi$ is not necessary an algebra homomorphism.} Since  $Gr_{{G}}(M)$ is freely generated by $\bar m_{1},\dots, \bar m_{k}$ over $Gr_{{F}}(A),$ this implies that $Gr_{{G}}(M)$ is freely generated by $\bar m_{1},\dots, \bar m_{k}$ over $Gr_{\bar{F}}(A).$

 \end{enumerate}

 \end{proof}

%
%
%
%

\section{Reduction to the Key Proposition}\label{Hirrevisited}
\setcounter{thm}{0}
In this section we prove Theorem \ref{thm:Intro.Hir}. We will need some notations:

\begin{itemize}
\item
Fix a natural number $n$. Let \RamiB{$H:=H_n:=SL_n$}.
\item
Let $E/F$ be an unramified quadratic extension of non-archimedean local fields of characteristic diffent than $2$.
\item  We let $\tau:E \to E$ be the Galois involution.

\item \EitanC{Let
$G=G_n:=Res^{E}_{F}(H_n)$ be the restriction of scalars of $\RamiB{H}$ to $E$ (in particular  $G(F)=\RamiB{H}(E)$).
}
\item \RamiB{We also fix} $X:=X_n$ the natural algebraic variety s.t. $X(F)=\{x \in G(E)|\tau(x^{t})=x\}$.
\item Let  $G$ act on $X$ by $$g \cdot x=gx\tau(g^{t})\RamiB{.}$$
\item
Let $D \subset X$ be the subset of diagonal matrices.
\item \EitanC{Finally, we let $T \subset G$ be the standard torus}.
\end{itemize}

In the above notations, Theorem \ref{thm:Intro.Hir} reads as follows:

\begin{theorem}\label{thm:Hir}
The module $\Sc(\Rami{X(F)})^{K_0}$ is free of rank $2^{n-1}$ over $\cH(G,K_0)$ where $K_0:=SL(n,\mathcal{O}_{E})$ is the standard maximal open subgroup of $G(F)$.

\end{theorem}



\begin{notation}
$ $
\begin{itemize}
\item $\pi$ a uniformizer in $\mathcal{O}_{E}$.
\item $q_{F}=|O_{F}/P_{F}|$, $q_{E}=|O_{E}/P_{E}|$.

\item $\Lambda$ the weight lattice of $G.$ We identify it with \RamiD{$\{(\lambda_1,\dots,\lambda_n)\in\mathbb{Z}^{n}|\lambda_1+\cdots+\lambda_n=0\}$}.
\item $\Lambda^{+}=\{\lambda=(\lambda_1, \dots, \lambda_n) \in \Lambda|\sum_{i=1}^k \lambda_i \geq 0 \,\, \, \forall k=1, \dots, n \}.$
\item $\Lambda^{++}=\{\lambda=(\lambda_1, \dots, \lambda_n) \in \Lambda| \lambda_k-\lambda_{k-1} \leq 0 \,\, \, \forall k=2,\dots, n \}.$ \RamiD{Note that $\Lambda^{++}\subset \Lambda^{+}.$}
\item for $\lambda \in \Lambda$ we set \RamiB{$\pi^{\lambda}:=\lambda(\pi) \in G(F).$}
\item  for $\lambda \in \Lambda$ we set \RamiB{$x_{\lambda}$ to be $\lambda(\pi)$ considered as an element in  $X(F).$}
\item Let $a_{\lambda}=e_{K_{0}\delta_{\pi^\lambda} K_{0}} \in \mathcal{H}(G,K_{0}).$
\item Let $m_{\lambda}=e_{K_{0}\delta_{x_\lambda}} \in \Rami{\Sc(X(F))}^{K_{0}}.$
\item We denote $\lambda \geq \lambda'$ iff $\lambda-\lambda' \in \Lambda^{+}.$ In this case, if $\lambda \ne \lambda'$ we denote $\lambda >\ \lambda'.$

\end{itemize}
\end{notation}

\RamiD{The following lemma is well known}\footnote{\RamiD{Part (1) is the  classical Cartan decomposition $G=K_{0}A^{++}K_{0}.$} A version of part (2) is proven in \cite{Jac}.}
\begin{lemma}\label{lemma: representatives}
$ $
\begin{enumerate}
\item The collection  $\{\pi^{\lambda}| \lambda \in \Lambda^{++}\}$ is \RamiB{a} complete set of representatives for the orbits of $K_{0} \times K_{0}$ on $G$.
\item The collection $\{x_{\lambda}| \lambda \in \Lambda^{++}\}$ is \RamiB{a} complete set of representatives for the orbits of $K_{0}$ on $X$.
\end{enumerate}
\end{lemma}

%
\begin{corollary}
$ $
\begin{enumerate}
\item The collection $\{a_{\lambda}| \lambda \in \Lambda^{++}\}$ is a basis for $\mathcal{H}(G,K_{0})$.
\item The collection $\{m_{\lambda}| \lambda \in \Lambda^{++}\}$ is a basis for $\Rami{\Sc(X(F))}^{K_{0}}$.
\end{enumerate}
\end{corollary}

This Corollary leads naturally to the following filtration on \RamiB{the} module $M:=\Sc(\Rami{X(F)})^{K_0}$ and the Hecke algebra $A:=\RamiB{\cH}(G,K_{0}).$

\begin{definition}\label{def:flts}
$ $
For $\lambda \in \Lambda$ we  introduce the subspaces
\begin{itemize}
\item $F_{\leq \lambda}(A)=Span_{\cc}\{a_{\mu}| \mu \leq \lambda \ ;  \mu \in  \Lambda^{++} \},\quad$
$F_{<\lambda}(A)=Span_{\cc}\{a_{\mu}| \mu < \lambda\}$
\item $\EitanC{G}_{\leq \lambda}(M)=Span_{\cc}\{m_{\mu}| \mu \leq \lambda \ ; \mu \in \Lambda^{++} \},\quad$
$\EitanC{G}_{<\lambda}(M)=Span_{\cc}\{m_{\mu}| \mu < \lambda\}$
\end{itemize}
\end{definition}

With this filtration we have the following Key Proposition:
\begin{proposition}\label{prop: maincomp}
$ $

\begin{enumerate}
\item  For every $\lambda \in \Lambda^{++}$ and $\mu \in \Lambda^{++}$ there exists a non-zero $p(\lambda,\mu) \in \cc$ such that $$a_{\lambda}a_{\mu}=p(\lambda,\mu)a_{ \lambda+\mu}+\RamiB{r}$$ with $\RamiB{r} \in F_{< \lambda+\mu}(A)$.

\item For every $\lambda \in \Lambda^{++}$ and $\mu \in \Lambda^{++}$ there exists a non-zero $q(\lambda,\mu) \in \cc$ and we have $$a_{\lambda}m_{\mu}=q(\lambda,\mu)m_{2 \lambda+\mu}+\RamiB{\delta}$$ where $\RamiB{\delta} \in \EitanC{G}_{<2 \lambda+\mu}(M)$.
\end{enumerate}
\end{proposition}
\RamiF{Part (1) is well known (see e.g. \cite[Chapter 5 (2.6)]{Mac}). We postpone the proof of Part (2) to} \S \ref{subsection: proof of prop} and continue with the proof of Theorem \ref{thm:Hir}

\begin{proof}[Proof of Theorem  \ref{thm:Hir}]
\RamiD{For $\lambda \in \Z^{n-1}$ denote $\tilde{F}_\lambda(A)=F_{\leq \tau(\lambda)}(A)$, $\tilde{\EitanC{G}}_\lambda(M)=\EitanC{G}_{\leq \tau(\lambda)}(M)$, where $$\tau((\lambda_1,\dots,\lambda_{n-1}))=(\lambda_1,\lambda_2-\lambda_1,\dots,\lambda_{n-1}-\lambda_{n-2},-\lambda_{n-1}).$$  Let $\phi:\Z^{n-1}\to \Z^{n-1}$ be given by $\phi(\lambda)=2\lambda$. \EitanC{P}roposition \ref{prop: maincomp} implies that  $\tilde F$ gives a structure of $\Z^n$-filtered algebra on $A$ and $\phi$-filtered module on $M$.}

Applying \RamiD{Proposition} \ref{grfree} it is enough to show that $Gr_{\EitanC{G}}(M)$ is finitely generated free $Gr_{F}(A)$-module. We now let $\bar{a}_{\lambda},\bar{m}_{\lambda}$ be the reductions of $a_{\lambda},m_{\lambda}$ to the associated graded. By proposition \ref{prop: maincomp} we get $\bar{a}_{\lambda}\bar{a}_{\mu}=p(\lambda,\mu)\bar{a}_{\lambda+\mu}$ and $\bar{a}_{\lambda}\bar{m}_{\mu}=q(\lambda,\mu)\bar{m}_{2\lambda+\mu}.$ 
Let $L \subset \Lambda^{++}$ be a such that $\Lambda^{++}=\cup_{\ell \in L}(\ell+2\Lambda^{+_{}+})$ is a disjoint covering.
Clearly, the set $\{m_{\ell}| \ell \in L\}$ is a free basis of $Gr_{\EitanC{G}}(M)$\ over $Gr_{F}(A)$. This finishes the proof.
\end{proof}
\section{Proof of Key Proposition \ref{prop: maincomp}}\label{subsection: proof of prop}
The proof of the proposition require an explicit version of Lemma \ref{lemma: representatives}. For this we require a definition.

\begin{definition}
Let $V=E^n$ and $V_0=F^n$
\begin{enumerate}
\item If $L_{1},L_{2}$ are two $O_E$-lattices in $V$ then we define
$$[L_{1}:L_{2}]=\log_{q_E}(|L_{1}/(L_{1} \cap L_{2})||L_{2}/(L_{1} \cap
L_{1})|^{-1})$$
\item
 Let $Q$ be a Hermitian form on $V.$ Let $L \subset V_0$ be a lattice.
  Take an $O_{F}$ basis $B=\{v_{1},\dots,v_{n}\}$ to $L.$ We define
$$\nu_{L}(Q)=\nu(det(Gram(B))):=\nu(det(Q(v_{i},v_{j}))),$$ \RamiB{where $\nu$ is the valuation of $E$.} This is
independent of the choice of the basis.


\end{enumerate}

\end{definition}


\begin{lemma}\label{lemma: explicitrepresentatives}
Let $\lambda=(\lambda_{1},\dots,\lambda_{n}) \in \Lambda^{++}$ and
denote by $p_k=\lambda_1+ \lambda_2+ \dots +\lambda_k$
and let $q_k=\lambda_n+ \lambda_{n-1}+ \dots +\lambda_{n-k+1}.$
\begin{enumerate}


\item Let $g \in K_{0} \pi^{\lambda}K_{0}.$ Then $p_k=\min \limits_{W \in Grass(k,V)} [W \cap O_{E}^{n} : W \cap gO_{E}^{n}].$

\item Let $x\in K_{0}x_{\lambda}.$
Then $q_k=\min \limits_{W \in Grass(k,V_{})} \nu_{O_{E}^{n} \cap
W}(x|_{W}).$


\end{enumerate}
\end{lemma}

\begin{proof}
\begin{enumerate}
\item We first note

$$\min \limits_{W \in Grass(k,V)} [W \cap O_{E}^{n} : W \cap gO_{E}^{n}]=\min \limits_{W \in Grass(k,V)} [W \cap O_{E}^{n} : W \cap \pi^{\lambda}O_{E}^{n}]$$

It remains to verify the statement of the lemma for $g=\pi^{\lambda}.$
Clearly,

 $$p_{k} \geq \min \limits_{W \in Grass(k,V)} [W \cap O_{E}^{n} : W \cap \pi^{\lambda}O_{E}^{n}]$$

Thus it is enough to show that for any $W \in Grass(\RamiB{k},V)$ we have
$$p_{k} \leq  [W \cap O_{E}^{n} : W \cap \pi^{\lambda}O_{E}^{n}]$$

For this we let $e_1,..,e_k$ be an $O_E$ basis for $W \cap O_E^{n}.$ Let $A \in Mat_{n \times k}(O_E)$ be the matrix whose $i$-the column is $e_i$, $i=1,..,k.$

Denote by $r(A)$ the matrix obtained from $A$ by reducing its elements to $O/\pi.$ Since $e_1,...,e_k$ is a basis we have $rank(r(A)) \geq k$ and we can find a $k \times k$ minor which is invertible in $O_E.$
Explicitly, we have  $\RamiB{\mathcal{I}}=(i_1,i_2,..,i_k)$ such that the minor $M_{\RamiB{\mathcal{I}},[1,k]}(A) \in O^{\times}.$

Notice that
\begin{multline*}[W \cap O_{E}^{n} : W \cap \pi^{\lambda}O_{E}^{n}]=[Span_{O_{E}}(e_1,..,e_k):
\pi^{\lambda}(\pi^{-\lambda}W \cap O_{E}^{n})]=\\=[Span_{O_{E}}(\pi^{-\lambda}e_1,..,\pi^{-\lambda}e_k):\pi^{-\lambda}W \cap O_{E}^{n}]=\\=[Span_{O_E}(\pi^{-\lambda}e_1,..,\pi^{-\lambda}e_k):Span_{E}(\pi^{-\lambda}e_1,..,\pi^{-\lambda}e_k) \cap O_{E}^{n}]
\end{multline*}
Let $f_1,...,f_k$ be an $O_E$-basis for $Span_{E}(\pi^{-\lambda}e_1,..,\pi^{-\lambda}e_k) \cap O_{E}^{n}.$ Let $B \in Mat_{n \times k}(O_E)$ be the corresponding matrix as before.


Let $C \in Mat_{k \times k}(E)$ be such that $B=\pi^{-\lambda}AC.$
Passing to the sub-matrix $B_{\RamiB{\mathcal{I}},[1,..,k]}$ we have $B_{\RamiB{\mathcal{I}},[1,..,k]}=diag(\pi^{-\lambda_{i_1}},...,\pi^{-\lambda_{i_k}})A_{\RamiB{\mathcal{I}},[1,..,k]}C.$
Thus $M_{\RamiB{\mathcal{I}},[1,k]}(B)=\pi^{-\sum_{j=1}^k\lambda_{i_j}}M_{\RamiB{\mathcal{I}},[1,k]}(A)det(C).$ Thus
$$0\leq\nu(M_{\RamiB{\mathcal{I}},[1,k]}(B))=-\sum_{j=1}^k\lambda_{i_j}+\nu(M_{\RamiB{\mathcal{I}},[1,k]}(A))+\nu(det(C))=-\sum_{j=1}^k\lambda_{i_j}+\nu(det(C))$$

Finally,
 \begin{multline*}[W \cap O_{E}^{n} : W \cap \pi^{\lambda}O_{E}^{n}]=[Span_{O_{E}}(\pi^{-\lambda}e_1,...,\pi^{-\lambda}e_k):Span_{O_E}(f_1,...,f_k)]=\nu(det(C))\geq\\ 
\geq \sum_{j=1}^k\lambda_{i_j} \geq p_k
 \end{multline*}


\item as before, the only non-trivial part is to show that
 $$\nu_{O_{E}^{n} \cap
W}(x_\lambda|_{W}) \geq q_k.$$
If $x_\lambda|_{W}$ is degenerate this is obvious.  So we will assume it is not. By Lemma \ref{lemma: representatives} we can find a $x_\lambda|_{W}$-orthonormal basis $(e_1, \dots, e_k)$ of ${O_{E}^{n} \cap
W} $ and a $x_\lambda|_{W^\bot}$-orthonormal basis $(e_{k+1}, \dots, e_n)$ of ${O_{E}^{n} \cap
{W^\bot}}$. Let $\mu_i=\tau(e_i^t)x_\lambda e_i$. By Lemma \ref{lemma: representatives} the collection $(\mu_1,\dots  ,\mu_n)$  coincides (up to reordering)  with $(\lambda_1,\dots  ,\lambda_n)$ thus $$\nu_{O_{E}^{n} \cap
W}(x_\lambda|_{W})=\mu_1 +\cdots+\mu_k \geq \lambda_n +\cdots+\lambda_{n-k+1}=q_k$$
\end{enumerate}

\end{proof}

\begin{proof}[Proof of Proposition \ref{prop: maincomp} (2)]
Since $x_{2 \lambda+\mu} \in \pi^{\lambda}K_{0}x_{\mu}$, it is enough to show that
$\pi^{\lambda}K_{0}x_{\mu} \subset \bigcup_{\nu \leq 2 \lambda+\mu} K_{0}x_{\nu}.$ Let $x \in K_{0}x_{\mu}.$

By Lemma \ref{lemma: explicitrepresentatives}(2)  we have to show $$\min \limits_{W \in
Grass(i,V)} \nu_{W \cap O^{n}}(\pi^{\lambda} \cdot x|_{W})
\leq \sum_{j=n-i+1}^{n} (\mu_{j}+2 \lambda_{j}).$$ By Lemma  \ref{lemma: explicitrepresentatives} we have,

\begin{multline*}\min \limits_{W \in Grass(i,V)} \nu_{O^{n} \cap
W}(\pi^{\lambda} \cdot x|_{W})=\min \limits_{W \in Grass(i,V)}
\nu_{\pi^{\lambda}O^{n} \cap
\pi^{\lambda}W}(x|_{\pi^{\lambda}W})=\\ = \min \limits_{W \in
Grass(i,V)} \nu_{\pi^{\lambda}O^{n} \cap W}(x|_{W})=
\min \limits_{W \in Grass(i,V)} (2[O^{n} \cap
W:\pi^{\lambda}O^{n}\cap W]+\nu_{O^{n} \cap W}(x|_{W})) \leq \\
\leq2 \min \limits_{W \in Grass(i,V)}([O^{n} \cap
W:\pi^{\lambda}O^{n} \cap W])+\sum_{j=n-i+1}^{n} \mu_{j}=\sum_{j=n-i+1}^{n} (2 \lambda_{j}+\mu_{j}).
\end{multline*}

\end{proof}



\vspace{1cm}

\end{document}